\documentclass[10pt,twoside,english,reqno,a4paper]{amsart}

\usepackage[utf8]{inputenc}
\usepackage[english]{babel}
\usepackage[lined,boxed]{algorithm2e}
\usepackage{float} 
\usepackage{hyperref}

\usepackage{cite} 

\usepackage{amsmath,amsthm,amssymb,latexsym,enumerate,color,listings,caption,graphicx,varioref,amscd,xcolor,amsfonts,graphics,lineno,eucal,comment}
\usepackage{mathrsfs} 
\usepackage{bm}
\usepackage{indentfirst}
\usepackage{array}

\usepackage{amsmath,amsthm,amssymb,latexsym,enumerate,color,hyperref}
\usepackage{graphicx}
\usepackage{xcolor}
\usepackage{eucal}
\usepackage{comment}

\allowdisplaybreaks

\theoremstyle{plain}
\newtheorem{thm}{Theorem}[section]
\newtheorem{prop}[thm]{Proposition}
\newtheorem{lem}[thm]{Lemma}
\newtheorem{cor}[thm]{Corollary}
\newtheorem{rem}[thm]{Remark}

\numberwithin{equation}{section}

\newcommand{\RR}{\mathbb{R}}
\newcommand{\CC}{\mathbb{C}}
\renewcommand{\SS}{\mathbb{S}}
\newcommand{\FF}{\mathbb{F}}
\newcommand{\MM}{\mathbb{M}}

\newcommand{\re}{\mathop{\mathrm{Re}}}
\newcommand{\im}{\mathop{\mathrm{Im}}}

\newcommand{\loc}{\mathrm{loc}}

\newcommand{\ol}{\overline}
\newcommand{\pa}{\partial}

\newcommand{\lan}{\langle}
\newcommand{\ran}{\rangle}
\newcommand{\tr}{\mathop{\mathrm{tr}}}

\newcommand{\na}{\nabla}
\newcommand{\nr}{\Vert}

\newcommand{\al}{\alpha}
\newcommand{\be}{\beta}
\newcommand{\ga}{\gamma}
\newcommand{\Ga}{\Gamma}

\newcommand{\De}{\Delta}

\newcommand{\la}{\lambda}
\newcommand{\La}{\Lambda}

\newcommand{\si}{\sigma}

\newcommand{\zi}{\zeta}
\newcommand{\om}{\omega}
\newcommand{\Om}{\Omega}

\newcommand{\cC}{\mathcal{C}}

\newcommand{\cF}{{\mathcal F}}
\newcommand{\cG}{{\mathcal G}}

\newcommand{\cJ}{{\mathcal J}}

\newcommand{\cS}{{\mathcal S}}

\newcommand{\cX}{\mathcal{X}}

\newcommand{\ds}{\displaystyle}

\author[R. Durastanti]{Riccardo Durastanti}
\address[Riccardo Durastanti]{Dipartimento di Matematica ed Informatica ``U.~Dini'', Universit\`a di Firenze, viale Morgagni 67/A, 50134 Firenze, Italy}
\curraddr{Dipartimento di Matematica e Applicazioni ``R. Caccioppoli", Universit\`a degli Studi di Napoli Federico II, via Cintia, Monte S. Angelo, 80126 Napoli, Italy}
\email{riccardo.durastanti@unina.it}

\author[R. Magnanini]{Rolando Magnanini}
\address[Rolando Magnanini]{Dipartimento di Matematica ed Informatica ``U.~Dini'', Universit\`a di Firenze, viale Morgagni 67/A, 50134 Firenze, Italy}
\email{rolando.magnanini@unifi.it}
\urladdr{https://people.dimai.unifi.it/magnanini/}

\begin{document}

\title[Asymptotic characterizations of holomorphic functions]{Nonlinear asymptotic mean value characterizations of holomorphic functions}

\begin{abstract}
Starting from a characterization of holomorphic functions in terms of a suitable mean value property, we build some nonlinear asymptotic characterizations for complex-valued solutions of certain nonlinear systems, which have to do with the classical Cauchy-Riemann equations. From these asymptotic characterizations, we derive suitable asymptotic mean value properties, which are used to construct appropriate vectorial dynamical programming principles. The aim is to construct approximation schemes for the so-called contact solutions, recently introduced by N. Katzourakis, of the nonlinear systems here considered.
\end{abstract}




\maketitle

\section{Introduction}

The classical {\it mean value property} characterizes harmonic functions --- the solutions of the classical Laplace equation $\De u=0$. In fact,
a continuous function $u$ is harmonic in an open set $\Om\subseteq\RR^N$ if and only if
\begin{equation}
\label{mean-value-property}
u(x) =\frac1{|B_r(x)|}\int_{B_r(x)}u(y)\,dy \ \mbox{ or } \
u(x)=\frac1{|\pa B_r(x)|}\int_{\pa B_r(x)}u(y)\,dS_y,
\end{equation}
for every ball $B_r(x)$ with $\ol{B_r(x)}\subset\Om$. (Throughout this paper, we shall denote the boundary of $\Om$ by the letter $\Ga$. Also, we use single bars to denote both the Lebesgue measure of a measurable set and the surface measure of a surface in $\RR^N$.)
From \eqref{mean-value-property} one can derive most of the properties of harmonic functions, such as weak and strong maximum principles, smoothness, analyticity, Liouville's theorem, Harnack's inequality, and more. The \textit{exact mean value property} in \eqref{mean-value-property} may be weakened. Indeed it is sufficient that a continuous function $u$ satisfies the \textit{asymptotic mean value property}, i.e.
$$
\frac1{|B_r(x)|}\int_{B_r(x)}u(y)\,dy = u(x) + o(r^2) \ \mbox{ or } \
\frac1{|\pa B_r(x)|}\int_{\pa B_r(x)}u(y)\,dS_y = u(x) + o(r^2),
$$
as $r \to 0^+$ (see \cite{Ful, NeVe, Kuz2}).
\par
Recently, \eqref{mean-value-property} has been somewhat extended to nonlinear settings.
E. Le Gruyer \cite{LeG} noticed the relationship of the nonlinear mean defined by
\begin{equation}
\label{legruyer}
\frac12\,\max_{B_r(x)}u+\frac12\,\min_{B_r(x)}u
\end{equation}
with the \textit{absolutely minimizing Lipschitz extensions}, which in particular are $\infty$-harmonic functions, i.e. viscosity solutions of the $\infty$-Laplace equation, $\De_\infty u=0$. Here,  $\De_\infty$ denotes the (game theoretic) $\infty$-laplacian.  Extensions to the case of the $p$-Laplace equation (with $1<p\le\infty$) have been first developed in \cite{MPR1, MPR2, MPR3, LPS}, based on a linear combination of the mean values in \eqref{mean-value-property} and \eqref{legruyer}. In the nonlinear case,  the connection between (viscosity) $p$-harmonic functions and the relevant mean is obtained by adopting the notion of \textit{asymptotic mean value property} (in an appropriate viscosity sense) for a continuous function (see \cite{MPR1}, \cite{IMW}, for a definition).
The references \cite{HR, RV} contain a two-dimensional approach to the case $p=1$. The relevant nonlinear mean value chosen is the \textit{median} of a function on $B_r(x)$, which is defined as the number $\mu$ such that
\begin{equation}
\label{median}
\bigl|\{ x\in B_r(x): u(x)>\mu\}\bigr|=\bigl|\{ x\in B_r(x): u(x)<\mu\}\bigr|;
\end{equation}
$\mu$ is unique if $u$ is continuous (see \cite{No}).

We also mention that, over the last years, there has been increasing interest on asymptotic mean value formulas for nonlinear elliptic equations. Here, we refer to: \cite{GiSm, AlSm, ArrPar, Al}, for the game-theoretic Laplacian; \cite{dTLin} for the variational $p$-Laplacian; \cite{Fe, BuSq, dTMeOc, FjNyWa}, for the fractional $p$-Laplacian; \cite{MeZh}, for double-phase operators; \cite{BCMR, Kuz} for general second order elliptic equations and Helmholtz equations; \cite{BCMR-JCA}, for the Monge-Amp\`ere equation; \cite{AKE1, AKE2, MiTe} in metric measure spaces.
\par
In \cite{IMW}, the second author of this paper and his co-authors proposed to use a \textit{natural} version of mean value. This is what we call in \cite{CIMW} a \textit{variational $p$-mean}, because it  is defined as the minimum point $\mu_p^r(u)(x)$ of the function
\begin{equation}
\label{variational}
\RR\ni\mu\mapsto \nr u-\mu\nr_{p,B_r(x)},
\end{equation}
where $\nr \cdot\nr_{p,B_r(x)}$ denotes the usual norm in the Lebesgue space $L^p(B_r(x))$. It is not difficult to show that the means defined by \eqref{mean-value-property}, \eqref{legruyer}, \eqref{median} coincide with $\mu_p^r(u)(x)$, in the corresponding cases. Moreover, as the means used before, the variational $p$-mean enjoys the following asymptotic property:
at each point $x\in\Om$ it holds that
\begin{equation}
\label{asymptotic-formula}
\mu^r_p(\phi)(x)=\phi(x)+c_{N,p}\,\De_p^G\phi(x)+o(r^2) \ \mbox{ locally uniformly in $\Om$ as $r\to 0^+$,}
\end{equation}
for every $\phi\in C^2(\Om)$ such that $\na\phi(x)\ne 0$ in $\Om$.
Here,  $c_{N,p}$ is some positive constant only depending on $N$ and $p$ and $\De_p^G$ denotes the \textit{game-theoretic $p$-Laplace operator,} formally defined by
$$
\De_p^G\phi=\De\phi+(p-2)\,\frac{\lan \na^2 \phi\, \na\phi, \na\phi\ran}{|\na\phi|^2}.
$$
 \par
As suggested above, the variational $p$-mean $\mu_p^r(u)$ has natural features that other types of means do not always have. Thanks to three of these, $\mu_p^r(u)$ results in an \textit{average}, according the definition given in \cite{TMP}.  In fact, as an operator on bounded functions, $\mu_p^r$ is: (i) \textit{affine invariant}, i.e. $\mu_p^r(\la\,\phi+c)=\la\,\mu_p^r(\phi)+c$, for $\la>0$ and $c\in\RR$; (ii) \textit{stable}, i.e. $\mu_p^r(\phi)$ is bounded from below and above
by the infimum and supremum of $\phi$, respectively; (iii) \textit{monotone}, i.e.
$\mu_p^r(\phi)\le\mu_p^r(\psi)$, if $\phi\le\psi$ pointwise.
Another important feature of a variational $p$-mean is that the functional $L^p(B_r(x))\ni u\mapsto\mu_p^r(u)\in\RR$ is continuous in the  $L^p(B_r(x))$-topology. These properties are naturally inherited from the variational definition of $\mu_p^r$, which can also be interpreted as the $L^p$-projection of $u$ on the subspace  of $L^p(B_r(x))$ of constant functions.
\par
As shown in \cite{TMP}, being an average is sufficient to prove the convergence of the dynamic programming principle (DPP)
\begin{equation}
\label{dpp}
u^r=\mu_p^r(u^r) \ \mbox{ on } \ \Om,
\end{equation}
subject to some appropriate boundary assumptions of Dirichlet type.
Two types of  effective boundary assumptions can be found in the literature. The former assigns a continuous extension on a strip around $\Om$ of a given Dirichlet data on $\Ga$ (see \cite{MPR3}, \cite{TMP}); the latter only assigns a Dirichlet data on $\Ga$, at the cost of slightly modifying the operator $\mu_p^r$ (see \cite{HR, HR2, RV, CIMW}).
\par
A function $u^r$ that, for a fixed $r>0$, satisfies \eqref{dpp} locally is also said to be \textit{$p$-harmonious} (it is clear that $r$ should be sufficiently small). Under appropriate regularity assumptions on the boundary $\Ga$ of $\Om$, one can prove that $p$-harmonious functions with a given data on $\Ga$ converge as $r\to 0^+$ to a $p$-harmonic function with the appropriate Dirichlet data on $\Ga$. The regularity assumptions differ depending on the boundary condition chosen for the DPP (see \cite{MPR3, HR3, CIMW, TMP}).
\par
The aim of this paper is to show further that the idea of defining an average variationally is very flexible, since it can be adapted to a variety of situations. One way to proceed is to replace the $L^p$-space in \eqref{variational} by other measure spaces, according to convenience. This  was indicated in \cite{IMW} in order to treat solutions of the \textit {game theoretic} parabolic $p$-Laplace equation. There, the Euclidean ball $B_r(x)$ equipped with the Lebesgue measure is replaced by the standard  \textit{heat ball} equipped with a suitable time-varying finite measure. By changing the relevant measure space, extensions of this scheme have been given in \cite{MS, DMRS}, for the Heisenberg group, and in \cite{AKPW} for Carnot groups. (See also \cite{FePi, PaPo, FeFo} for non-variational mean value formulas in Carnot groups.)
\par
In this paper, we begin to investigate on how the variational framework can be extended to vector-valued functions $\phi$ from $\Om\subset\RR^N$ to $\RR^N$. We present our ideas having in mind possible applications to nonlinear systems of partial differential equations.
Here, we begin our investigation by considering an important case study, which entails holomorphic functions of one complex variable.
\par
In order to clarify our plan, let us consider the case in which $N=p=2$. It is easily seen that, if $\phi=(\phi_1,\phi_2)\in C^2(\Om;\RR^2)$ and $\mu^r(\phi)(x)\in\RR^2$ minimizes on $\RR^2$ the function
$$
\RR^2\ni \mu\mapsto \int_{B_r(x)} |\phi(y) -\mu|^2 dy,
$$
then we obtain the asymptotic formula:
$$
\mu^r(\phi)(x)-\phi(x)=\frac18\,\De\phi(x)\,r^2+o(r^2) \ \mbox{ as } \ r\to 0^+.
$$
Here, we mean $\De \phi=(\De\phi_1, \De\phi_2)$.
\par
Thus, we obtain that $\phi$ is a \textit{harmonic map} from $\Om$ to $\RR^2$, i.e. both functions $\phi_1$ and $\phi_2$ are harmonic in $\Om$ if and only if $\mu^r(\phi)(x)-\phi(x)=o(r^2)$ as $r\to 0^+$ for every $x\in\Om$. In other words, the smooth solutions of the system
$$
\De \phi_1=0, \quad \De \phi_2=0 \ \mbox{ in } \ \Om
$$
are characterized by an asymptotic mean value property (as a matter of fact, by even an \textit{exact} mean value property). With a little more effort, the smooth solutions of the system
$$
\De^G_p \phi_1=0, \quad \De^G_p \phi_2=0 \ \mbox{ in } \ \Om,
$$
i.e. the \textit{$p$-harmonic maps},
can be characterized (away from their critical points) by the asymptotic mean value property $\mu_p^r(\phi)(x)-\phi(x)=o(r^2)$ as $r\to 0^+$, where $\mu_p^r(\phi)(x)\in\RR^2$ is the unique minimum point on $\RR^2$ of the function
$$
\RR^2\ni(\mu_1, \mu_2)\mapsto \int_{B_r(x)}\bigl[|\phi_1(y)-\mu_1|^p+|\phi_2(y)-\mu_2|^p\bigr] dy.
$$
\par
In both cases, however, we obtain a characterization of solutions of an \textit{uncoupled system} of differential equations --- a slight generalization of the scalar case.
Instead, the aim of our research is to obtain a non-trivial connection between some sort of mean value property (exact or asymptotic) and the solutions of some coupled system of partial differential equations. With this in mind, as a case study, we shall analyse in this paper the Cauchy-Riemann system --- maybe the most studied \textit{coupled} system of partial differential equations---  and some of its  possible nonlinear generalizations.
\par
Thus, we shall consider complex-valued functions $f=u+i\,v$ of the complex variable $z=x+i\,y$ on subdomains of the complex plane $\CC$. Also, we will denote by $D_r(z)$ the disk in $\CC$ centered at $z$ and with radius $r>0$ and we set $D_r=D_r(0)$, $D=D_1$ and $\cS=\pa D$. We know that sufficiently regular functions $f=u+i\,v$, which satisfy the Cauchy-Riemann system,
$$
v_x=-u_y, \quad v_y=u_x,
$$
or in complex notation
$$
f_{\ol{z}}=0,
$$
are the so-called \textit{analytic} or \textit{holomorphic} functions. (Here, $\ol{z}=x-i y$ is the complex conjugate of $z$, so that $2\,f_{\ol{z}}=f_x+i\,f_y$, $2\,f_z=f_x-i\,f_y$.)  We also know that any holomorphic function is also harmonic, i.e. it satisfies the mean value properties
\begin{equation*}
\label{mean-value-holomorphic}
f(z)=\frac1{|D_r(z)|}\int_{D_r(z)} f(\zi)\,dA_\zi \ \mbox{ and } \ f(z)=\frac1{|\pa D_r(z)|}\int_{\pa D_r(z)} f(\zi)\,dS_\zi,
\end{equation*}
for any disk $D_r(z)$ with $\ol{D_r(z)}\subset\Om$. Here, $dA_\zi$ and $dS_\zi$ denote the respective volume and surface elements. Conversely, there are (complex-valued) harmonic functions which are not holomorphic: their real and imaginary parts are both harmonic, but they do not solve the Cauchy-Riemann system (see also \cite{BaiWu}).
Hence, if we want a characterization of holomorphic functions in terms of a mean value property and, more generally, its extension to nonlinear systems related to the Cauchy-Riemann equations, we must turn to another type of mean. The starting point of our study is the following characterization, which is an adaptation of classical arguments. (We shall postpone its proof to Appendix A.)
\begin{prop}[Mean value characterization of holomorphic functions]
\label{th:new-mean-value}
Let $\Om\subseteq\CC$ be a 
domain and let $f:\Om\to\CC$ be a continuous function.
Then, the function $f$ is of class $C^1(\Om)$ and holomorphic in $\Om$ if and only if
\begin{equation}
\label{new-mean-value}
f(z)=\frac1{|D_r(z)|}\int_{D_r(z)} f(\zi)\left[1+\frac{2}{r}\,(\zi-z)\right] dA_\zi,
\end{equation}
for any disk $D_r(z)$ with $\ol{D_r(z)}\subset\Om$.
\end{prop}
\par
In Section \ref{sec:mean-value-property} we reformulate this result in a suitable way, and starting from Remark \ref{rem:projection}, we then build up our investigation on (nonlinear) asymptotic mean value properties for complex-valued solutions of certain nonlinear systems, which have to do with the classical Cauchy-Riemann equations. In fact, in Section \ref{sec:nonlinear-characterizations},
we shall consider the following general situation.
\par
Let $F:[0,\infty)\to[0,\infty)$ be a function of class $C^1([0,\infty))\cap C^2((0,\infty))$, which is strictly convex and such that $F(0)=F'(0)=0$
and $F(s)\to+\infty$ as $s\to+\infty$.
Let $\Om\subset\CC$ be an open domain and take any closed disk $\ol{D_r(z)}\subset\Om$. Next, for $f\in L_\loc^1(\Om; \CC)$ and any $c\in\CC$ define:
\begin{equation}
\label{function-c}
\cF(c)=\int_{\pa D_r(z)} F\left(\left|f(\zi)-c\,\ol{(\zi-z)}\right|\right) dS_\zi.
\end{equation}
Under these assumptions, it is easy to see that $\cF$ has exactly one minimum point $c^\cF(f,r)(z)$ on $\CC$. With these premises, we present our main result: a characterization of solutions of a nonlinear system of Cauchy-Riemann type.

\begin{thm}
\label{th:asymptotics-F}
Let $F\in C^1([0,+\infty))\cap C^2((0,\infty))$ be a strictly convex function such that $F(0)=F'(0)=0$.
\par
Set
\begin{equation}
\label{Lambda}
\La(s)=\frac{s\,F''(s)}{F'(s)} \ \mbox{ for } \ s>0
\end{equation}
and assume that, for some constants $0<\La^-\le\La^+$, it holds that
\begin{equation}
\label{Lambda-bounds}
\La^-\le\La(s)\le\La^+ \ \mbox{ if } \ s>0.
\end{equation}
\par
Let $\Om\subseteq\CC$ be an open set and let the function $f:\Om\to\CC$ be differentiable in $\Om$. Let $c^\cF(f,r)(z)$ be the minimum point of $\cF$ on $\CC$.
Then, away from the zeroes of $f$, we have that
$$
c^\cF(f,r)=o(1) \ \mbox{ as } \ r\to 0^+
$$
if and only if
$f$ satisfies the equation
\begin{equation}
\label{general-eq-f}
f_{\ol{z}}+\frac{\La(|f|)-1}{\La(|f|)+1}\,\frac{f}{\ol{f}}\,\ol{f}_{\ol{z}}=0 \ \mbox{ in } \ \Om.
\end{equation}
\end{thm}
The proof of this result will be given in Section \ref{subsec:variational-means}. The assumptions on $F$ are quite general and include the case of the so-called $p$-means for $p>1$, in which $F(s)=s^p/p$. In Corollary \ref{th:asymptotics-infty}, we will also carry out the case $p=\infty$, in which the integral in \eqref{function-c} is replaced by the supremum of $|f(\zi)-c\,\ol{(\zi-z)}|$ for $\zi\in \partial D_r(z)$.
\par
As a by-product of Theorem \ref{th:asymptotics-F}, we obtain the following nonlinear asymptotic characterization of holomorphic functions.

\begin{thm}
\label{th:nonlinear-holomorphic}
Let $F: [0,\infty)\to [0,\infty)$ satisfy the assumptions of Theorem \ref{th:asymptotics-F} and denote by $G: [0,\infty)\to [0,\infty)$ the Young conjugate of $F$.
\par
Suppose that $g$ is differentiable in $\Om$ and let $c^\cJ(g,r)(z)$ be the minimum point in $\CC$ of the function:
\begin{equation}
\label{defJ}
\cJ(c)= \int_{\pa D_r(z)} F\left(\left|\frac{G'(|g(\zi)|)}{|g(\zi)|}g(\zi)-c\,\ol{(\zi-z)}\right|\right) dS_\zi, \ \ c\in\CC.
\end{equation}
Then, away from the zeroes of $g$, we have that $g$ is holomorphic in $\Om$ if and only if
$c^\cJ(g,r)(z)=o(1)$ as $r\to 0^+$.
\end{thm}
The definition and properties of the Young conjugate $G$ will be recalled in Section \ref{subsec:nonlinear-holomorphic}, together with the proof of the last theorem.
\par
The last issue that we address in this paper is an attempt to develop approximation algorithms for solutions of nonlinear systems of partial differential equations by solutions of certain vectorial dynamic programming principles (DPP) (see Section \ref{sec:classical-AMVP}). As already mentioned, such algorithms have been constucted in the scalar case, by solving the DPP  \eqref{dpp} --- based on some nonlinear mean value properties --- for a fixed radius $r$ and then by letting $r$ tend to zero. In the limiting process, the theory of viscosity solutions and the concept of asymptotic mean value property in the viscosity sense play a central role.
\par
The aim of Sections \ref{sec:classical-AMVP} and \ref{subsec:contact-amvp} is to extend this scheme to the vectorial case. In order to do this, we must by-pass at least two obstructions.
The first one has to do with setting up a suitable vectorial DPP, which has a fixed-point structure as in \eqref{dpp}, based on the characterizations we carried out in Section \ref{sec:nonlinear-characterizations}. However,
those characterizations do not produce the desired DPP structure.
In fact, for instance, if we were to use the weighted mean $c^\cF(f,r)$, the approximating equation for fixed $r$ should be $c^\cF(f,r)=0$.

Therefore, in Section \ref{sec:classical-AMVP} we modify the means $c^\cF(f,r)$ in the spirit of Proposition \ref{th:new-mean-value}, Corollary \ref{cor:new-mean-value}, and Remark \ref{rem:projection}. These results suggest that a DPP structure can be set up, at least in the quadratic case. In fact, with Remark \ref{rem:projection} in mind, we consider the function defined by
\begin{equation*}
\cG(a,b)=\int_{\pa D_r(z)} \left[F\bigl(|f(\zi)-a|\bigr)+ F\bigl(|f(\zi)-b\,\ol{(\zi-z)}|\bigr)\right] dS_\zi,
\end{equation*}
for $a, b\in\CC$. We then denote the unique minimum point of $\cG$ on $\CC\times\CC$ by $(a^\cG(f,r), b^\cG(f,r))$ and set
\begin{equation}
\label{def-mu}
\mu^\cG(f,r)(z)=a^\cG(f,r)(z)+r\,b^\cG(f,r)(z).
\end{equation}
The following theorem provides the basis to construct an appropriate DPP. In analogy with what done in \cite{IMW}, \cite{CIMW}, we say that a differentiable function $f:\Om\to\CC$ satisfies the \textit{asymptotic mean value property (AMVP)} in $\Om$ if
\begin{equation*}
\label{asymptotics-G}
\mu^\cG(f,r)(z)-f(z)= o(r) \ \mbox{ as } \ r\to 0^+,
\end{equation*}
for any $z\in\Om$.
\begin{thm}[Characterization by AMVP]
\label{th:dpp-theorem}
Let $\Om$ be a domain in $\CC$ and let $f$ be differentiable in $\Om$. Let $\mu^\cG(f,r)(z)$ be defined as in \eqref{def-mu}.
Then,
away from the zeroes of $f$ in $\Om$, $f$ satisfies the equation \eqref{general-eq-f}
if and only if it satisfies the asymptotic mean value property  in $\Om$.
\end{thm}
In Remark \ref{linkSec2} we explicitly link Proposition \ref{th:new-mean-value} to Theorem \ref{th:dpp-theorem}.
\par
As the AMVP suggests, it is clear that the desired vectorial DDP is
\begin{equation*}
f^r=\mu^\cG(f^r,r).
\end{equation*}
The problem of solving it with suitable boundary conditions will be considered in a forthcoming paper.
\par
Instead, we turn to the second obstruction. This has to do with the concept of solution $f$ of the system \eqref{general-eq-f} that one should use in order to carry out its approximation by the solutions $f^r$ of the DPPs for $r>0$.
Theorem \ref{th:dpp-theorem} (see Lemma \ref{lem:ab-theorem}, as well) suggests that at the zeroes of $f$ the concept of classical solution has to be given up and a solution of the system \eqref{general-eq-f} must be intended in some suitable generalized sense. The standard weak sense is hardly viable due to the fact that \eqref{general-eq-f} (and generally equations obtained as limits in a DPP approximation process) is not (genuinely) variational. The viscosity sense works for scalar elliptic and parabolic equations and is based on comparison principles. This feature prevents a straightforward extension to systems, though.
\par
Thus, in Section \ref{subsec:contact-amvp}, we attempt to adapt to our case study the theory of \textit{contact solutions}, which N. Katzourakis has proposed in \cite{Ka}, very recently. This tries to extend the theory of viscosity solutions for partial differential equations to the case of systems. The new theory may
appear quite intricate, having to do with tensor calculus. Nevertheless, it succeeds to effectively generalize several important features of that of viscosity solutions. The most significant are the definition of solutions by jets or by touching test functions and a stabilty theorem. In Section \ref{subsec:gen-contact-solutions}, we shall recall the details of the theory which are pertinent to our investigation.
In order to perform our adaptation, in Section \ref{subsec:CAMVP}, we introduce the idea of \textit{contact asymptotic mean value property (CAMVP)}, which is aimed to generalise the AMVP for viscosity solutions, introduced in \cite{MPR1}. Thus, with reference to the definitions contained in  Section \ref{subsec:contact-amvp}, we present our last characterization, which is in the spirit of those presented in \cite{MPR1} and \cite{IMW} for scalar functions.
\begin{thm}[Characterization by CAMVP]
\label{th:contact-asymptotic-characterization}
Let $\Om$ be a domain in $\CC$ and let  $f:\Om\to\CC$ be a continuos function in $\Om$.
Then,
$f$ is a contact solution of the system \eqref{general-eq-f}
if and only if it satisfies the contact asymptotic mean value property in $\Om$.
\end{thm}

Applications of the ideas presented in this paper are various. Mainly, they are attempts of extensions to systems of known results and techniques for scalar equations.  One already mentioned is the estension of DPPs to vectorial cases. Also, the approximation by means of solutions of suitable DPPs of solutions of systems of PDEs may help to obtain information about the regularity of the latter, which is a current field of intensive study (see e.g. \cite{LP, LPS} in the scalar case). Further, we notice that the relevant DPPs mentioned in this paper may be interpreted as some kind of short-range non-local equations.

\section{A mean value property for holomorphic functions}
\label{sec:mean-value-property}

In this section, we briefly explore on the connection between the mean value property or its asymptotic counterpart and the minimization of a suitable functional. This is the starting point of the analysis carried out in the remaining sections, in which we extend that connection to some nonlinear settings, thus giving a link among Proposition 1.1 and Theorems 1.2 and 1.4, whose proofs are presented in Sections \ref{sec:nonlinear-characterizations} and \ref{sec:classical-AMVP}.
\par
We start with a direct consequence of Proposition \ref{th:new-mean-value}.
\begin{cor}
\label{cor:new-mean-value}
Let $\Om\subseteq\CC$ be a domain and let $f:\Om\to\CC$ be a continuous function.
Then, the following assertions are equivalent:
\begin{enumerate}[(i)]
\item
the mean value formula \eqref{new-mean-value} holds
for any disk $D_r(z)$ with $\ol{D_r(z)}\subset\Om$;
\item
$f\in C^1(\Om)$ and the following asymptotic mean value formula formula holds at every $z\in\Om$:
\begin{equation*}
f(z)=\frac1{|D_r(z)|}\int_{D_r(z)} f(\zi)\left[1+\frac{2}{r}\,(\zi-z)\right] dA_\zi+o(r) \ \mbox{ as } \ r\to 0^+;
\end{equation*}
\item
$f$ is holomorphic in $\Om$.
\end{enumerate}
\end{cor}

\begin{proof}
Proposition \ref{th:new-mean-value} ensures that $(i)$ and $(iii)$ are equivalent. It is also clear that $(ii)$ follows from $(i)$. By an inspection of the arguments used in item $(i)$ of the proof of Proposition \ref{th:new-mean-value}, we easily infer that $(iii)$ follows from $(ii)$.
\end{proof}

Now we define the functional $\cG:\CC^2\to\RR$ as
$$
\cG(a,b)=\int_{D_r(z)}\left|f(\zi)-a-b\,\ol{(\zi-z)}\right|^2 dA_\zi \ \mbox{ for } \ a, b\in\CC.
$$
It is not difficult to show that $\cG$ has only one minimum point $(a(f,r),b(f,r))$ in $\CC^2$, and we have that
\begin{eqnarray*}
&&a(f,r)(z)=\frac1{|D_r(z)|}\int_{D_r(z)} f(\zi)\,dA_\zi,  \\
&&b(r,f)(z)=\frac2{r^2 |D_r(z)|}\int_{D_r(z)} f(\zi)\,(\zi-z)\,dA_\zi.
\end{eqnarray*}
Thus, the mean value at the right-hand side of \eqref{new-mean-value} equals
$$
a(f,r)(z)+r\,b(f,r)(z)=\pi^r(f)(z,z+r),
$$
where
$$
\pi^r(f)(z,\zi)= a(f,r)(z)+b(f,r)(z)\,\ol{(\zi-z)}
$$
is the $L^2$-projection of $f$ on the subspace of affine anti-holomorphic functions. As a direct consequence of Corollary \ref{cor:new-mean-value} we have the following remark which gives an idea of how to define the right variational mean, in the nonlinear case.

\begin{rem}
\label{rem:projection}
Let $(a(f,r),b(f,r))$ be a minimum point of
$$
\cG(a,b)=\int_{D_r(z)}\left|f(\zi)-a-b\,\ol{(\zi-z)}\right|^2 dA_\zi.
$$
Then $f$ is \textit{holomorphic} in $\Om$ if and only if
$$
a(f,r)(z)+r\,b(f,r)(z) = f(z) + o(r) \text{ as }r\to 0^+.
$$
\par
For the sequel, we finally observe that the function $\cG$ shares the same critical (minimum) point of the function defined by
$$
\int_{D_r(z)}\left[|f(\zi)-a|^2+|f(\zi)-b\,\ol{(\zi-z)}|^2\right] dA_\zi \ \mbox{ for } \ a, b\in\CC.
$$
\end{rem}

\section{A nonlinear characterizations of holomorphic functions}
\label{sec:nonlinear-characterizations}

In this section, we introduce some variational means on spheres for complex-valued functions related to quite general convex densities and we shall study their asymptotic behavior as the radius of the spheres tends to zero. This analysis will lead to characterizations of complex-valued solutions of certain nonlinear equations and, as a by-product, to a new nonlinear characterization of holomorphic functions.

\subsection{Variational means of complex-valued functions}
\label{subsec:variational-means}
For the reader's convenience, we recall some notations and definitions from the introduction. \par
Let $F\in C^1([0,\infty))\cap C^2((0,\infty))$ be a convex function with $F(0)=F'(0)=0$
and such that $F(s)\to+\infty$ as $s\to+\infty$.
Let $\Om\subset\CC$ be an open domain and let  $f\in L_\loc^1(\Om; \CC)$. For any $c\in\CC$ and any closed disk $\ol{D_r(z)}\subset\Om$, we consider the function in \eqref{function-c}, i.e.
\begin{equation*}
\cF(c)=\int_{\pa D_r(z)} F\left(\left|f(\zi)-c\,\ol{(\zi-z)}\right|\right) dS_\zi.
\end{equation*}
It is convenient to set $H(w)=F(|w|)$ and rewrite \eqref{function-c} as
\begin{equation}
\label{function-c-H}
\cF(c)=\int_{\pa D_r(z)} H\left(f(\zi)-c\,\ol{(\zi-z)}\right) dS_\zi.
\end{equation}

It is clear that a minimum of $\cF$ on $\CC$ exists.
Let $\cC$ be the set of minimum points of $\cF$ on $\CC$. Any $c\in\cC$ satisfies the equation:
\begin{equation}
\label{fermat-c}
\int_{\pa D_r(z)} H_{\ol{w}}\left(f(\zi)-c\,\ol{(\zi-z)}\right) (\zi-z)\,dS_\zi=0,
\end{equation}
since $\pa_{\ol{c}}\cF=0$ at $c$. Moreover, if $\cC$ contains only one point $c^\cF(f,r)(z)$, then the last equation characterizes $c^\cF(f,r)(z)$.

\begin{prop}
\label{lem:existence-uniqueness-minimum}
Let $F\in C^1([0,\infty))\cap C^2((0,\infty))$ be a strictly convex function such that  $F(0)=F'(0)=0$
and $F(s)\to+\infty$ as $s\to+\infty$.
Then, for any $r>0$, there exists a unique minimum point $c^\cF(f,r)(z)$ for $\cF$ on $\CC$, i.e. $\cC=\{c^\cF(f,r)(z)\}$.
\end{prop}
\begin{proof}
The strict convexity of $F$ and the growth condition make sure that $\cF$ has only one minimum point, i.e. $\cC$ is a singleton.
\end{proof}

In what follows, $\La(s)$ is the function defined in \eqref{Lambda}.

\begin{prop}
\label{prop:equation-for-cr}
Let $F\in C^1([0,\infty))\cap C^2((0,\infty))$ be a convex function such that  $F(0)=F'(0)=0$, $F(s)\to+\infty$ as $s\to+\infty$ and $F'(s)/s\in L^1((0,1))$. Denote by $\cS$ the unit circle.
\par
For $r>0$, set $c_r=c^\cF(f_a,r)(z)$, where $f_a$ is the affine function defined by
$$
f_a(\zi)=\om+\si\,(\zi-z)+\tau\,\ol{(\zi-z)} \ \mbox{ for } \ \zi\in\CC,
$$
with $\om\in\CC\setminus\{0\}$ and $\si, \tau\in\CC$.
Then, it holds that
\begin{equation}
\label{equation-for-cr}
c_r=\tau+\al(r)\,\ol{\si}+\be(r)\,\si+\ga(r)\,(\ol{\tau}-\ol{c_r}),
\end{equation}
where
\begin{eqnarray*}
&&\al(r)=\frac{\int_\cS\int_0^1 \frac{F'(|w|)}{|w|}\,[\La(|w|)-1]\,\frac{w}{\ol{w}}\,dt\,dS_\zi}{\int_\cS\int_0^1 \frac{F'(|w|)}{|w|}\,[\La(|w|)+1]\,dt\,dS_\zi},  \\
&&\be(r)=\frac{\int_\cS\int_0^1 \frac{F'(|w|)}{|w|}\,\bigl[\La(|w|)+1\bigr]\,\zi^2 dt\,dS_\zi}{\int_\cS\int_0^1 \frac{F'(|w|)}{|w|}\,[\La(|w|)+1]\,dt\,dS_\zi}, \\
&&\ga(r)=\frac{\int_\cS\int_0^1 \frac{F'(|w|)}{|w|}\,\bigl[\La(|w|)-1\bigr]\,\frac{w}{\ol{w}}\,\zi^2 dt\,dS_\zi}{\int_\cS\int_0^1 \frac{F'(|w|)}{|w|}\,[\La(|w|)+1]\,dt\,dS_\zi}.
\end{eqnarray*}
Here, we denote $w=w(r,t;\zi)=\om+t\,r\,[\si\,\zi+(\tau-c_r)\,\ol{\zi}]$ for $t\in[0,1]$.
\par
In particular, if $F(s)=s^p/p$ with $p>1$, we have that
$$
c_r=\tau+\frac{p-2}{p}\,\al_p(r)\,\ol{\si}+\be_p(r)\,\si+\frac{p-2}{p}\,\ga_p(r)\,(\ol{\tau}-\ol{c_r}),
$$
where
\begin{eqnarray*}
&\ds\al_p(r)=\frac{\int_\cS\int_0^1 |w|^{p-2} \frac{w}{\ol{w}}\, dt\,dS_\zi}{\int_\cS \int_0^1|w|^{p-2} dt\,dS_\zi} \\
&\ds\be_p(r)=\frac{\int_\cS \int_0^1|w|^{p-2} \zi^2 dt\,dS_\zi}{\int_\cS\int_0^1 |w|^{p-2} dt\,dS_\zi},
\quad\ga_p(r)=\frac{\int_\cS \int_0^1 |w|^{p-2}\frac{w}{\ol{w}} \zi^2 dt\,dS_\zi}{\int_\cS\int_0^1 |w|^{p-2} dt\,dS_\zi}.
\end{eqnarray*}
\end{prop}

\begin{proof}
By a change of variable in \eqref{fermat-c}, we have that
$$
\int_{\cS} H_{\ol{w}}\left(\om+r\,[\si\,\zi+(\tau-c_r)\ol{\zi}]\right) \zi\,dS_\zi=0,
$$
and hence that
$$
\frac1{r} \int_{\cS} \left\{H_{\ol{w}}\left(\om+r\,[\si\,\zi+(\tau-c_r)\ol{\zi}]\right)-H_{\ol{w}}\left(\om\right) \right\} \zi\,dS_\zi=0,
$$
being as $\int_\cS\zi\,dS_\zi=0$.
\par
Since $H(w)=F(|w|)$, for every $w\neq 0$ we compute that
\begin{eqnarray*}
&&H_{\ol{w}\hspace{.5pt} w}(w)=\frac{|w|\,F''(|w|)+F'(|w|)}{4 |w|}=\frac{F'(|w|)}{4 |w|}\bigl\{\La(|w|)+1\bigr\}, \\
&&H_{\ol{w}\hspace{.5pt} \ol{w}}(w)=\frac{|w|\,F''(|w|)-F'(|w|)}{4 |w|}\,\frac{w}{\ol{w}}=\frac{F'(|w|)}{4 |w|}\bigl\{\La(|w|)-1\bigr\}\,\frac{w}{\ol{w}}.
\end{eqnarray*}
By our assumptions both $F'(s)/s$ and $F''(s)$ belong to $L^1((0,1))$. Thus, we can apply the fundamental theorem of calculus and obtain that
\begin{multline*}
0=\int_\cS\int_0^1\frac{d}{dt}H_{\ol{w}}(\om+t\,r\,[\si\,\zi+(\tau-c_r)\ol{\zi}])\,\zi\,dt\,dS_\zi= \\
\int_\cS\int_0^1\left\{H_{\ol{w} w}\left(w\right)[\si\,\zi^2+\tau-c_r]+
 H_{\ol{w}\, \ol{w}}\left(w\right)[\ol{\si}+(\ol{\tau}-\ol{c_r})\,\zi^2]\right\} dt\,dS_\zi,
\end{multline*}
where we mean that $w=w(r,t;\zi)$ for $t\in[0,1]$, for notational convenience.
\par
Hence we get that
\begin{multline*}
\int_\cS\int_0^1 \frac{F'(|w|)}{|w|}[\si\,\zi^2+\tau-c_r][\La(|w|)+1]\,dt\,dS_\zi+ \\
\int_\cS \int_0^1\frac{F'(|w|)}{|w|} [\ol{\si}+(\ol{\tau}-\ol{c_r})\,\zi^2][\La(|w|)-1]\,\frac{w}{\ol{w}}\,dt\,dS_\zi=0,
\end{multline*}
after some algebraic manipulations. Then, \eqref{equation-for-cr} follows at once.
\par
The proof of the formula for $F(s)=s^p/p$ follows by easy computations.
\end{proof}

\begin{lem}
\label{lem:asymptotics-F}
Let $F\in C^1([0,\infty))\cap C^2((0,\infty))$ be a strictly convex function such that  $F(0)=F'(0)=0$.
Assume that the function $\La$ defined in \eqref{Lambda} satisfies \eqref{Lambda-bounds}.
\par
Let $f_a$ be the affine function defined by
$$
f_a(\zi)=\om+\si\,(\zi-z)+\tau\,\ol{(\zi-z)} \ \mbox{ for } \ \zi\in\CC,
$$
where $\om\in\CC\setminus\{0\}$ and $\si, \tau\in\CC$. Then, we have that
\begin{equation}
\label{asymptotics-c-affine}
c^\cF(f_a,r)(z)=\tau+\frac{\La(|\om|)-1}{\La(|\om|)+1}\,\frac{\om}{\ol{\om}}\,\ol{\si}
+o(1) \ \mbox{ uniformly as } \ r\to 0^+.
\end{equation}
\end{lem}

\begin{proof}
Notice that the assumptions on $F$ and $\La(s)$ give that the growth conditions
\begin{eqnarray*}
&&\frac{K}{1+\La^-}\,s^{1+\La^-}+K_-\le F(s)\le \frac{K}{1+\La^+}\,s^{1+\La^+}+K_+, \\
&&K\,s^{\La^-}\le F'(s)\le K\,s^{\La^+},
\end{eqnarray*}
hold for any $s\ge 1$, with $K=F'(1)$ and $K_-, K_+\in\RR$.
In particular, $F(s)\to+\infty$ as $s\to+\infty$ and $F'(s)/s\in L^1((0,1))$. Thus, we can apply Proposition \ref{prop:equation-for-cr}.
\par
For notational convenience set $c_r= c^\cF(f_a,r)(z)$. Next, we claim that, under our assumptions on $\La(s)$, we have that
$|c_r|$ remains bounded as $r\to 0^+$. In fact, we can easily infer that $|\be(r)|\le 1$ and
$$
|\al(r)|, |\ga(r)|\le \frac{\int_\cS \int_0^1\frac{F'(|w|)}{|w|}\,|\La(|w|)-1|\,dt\,dS_\zi}{\int_\cS \int_0^1 \frac{F'(|w|)}{|w|}\,[\La(|w|)+1]\,dt\,dS_\zi}\le\max\left[\frac{\La^+-1}{1+\La^+},\frac{1-\La^-}{1+\La^-}\right]<1,
$$
since
$$
|\La(s)-1|=[1+\La(s)]\,\frac{|\La(s)-1|}{1+\La(s)}\le \max\left[\frac{\La^+-1}{1+\La^+},\frac{1-\La^-}{1+\La^-}\right] [1+\La(s)].
$$
Thus, from \eqref{equation-for-cr} we infer that
$$
\left\{ 1-\max\left[\frac{\La^+-1}{1+\La^+},\frac{1-\La^-}{1+\La^-}\right]\right\} |c_r|\le 2\,(|\tau|+|\si|),
$$
for any $r>0$. Hence, up to subsequences, $c_r$ converges as $r\to 0$ to some complex number $c_0$.
\par
Now, in the limiting process as $r\to 0$, the growth conditions on $F'(s)$ allow to use the arguments (with $p=\La^+$) presented in the proof of \cite[Lemma 3.1]{IMW}, which are based on applications of the dominated convergence theorem and its generalized version.
Therefore, since $w(r)\to\om$ as $r\to 0$, with $\om\ne 0$, and $\int_\cS\zi^2 dS_\zi=0$, we can conclude that
$$
\al(r)\to \frac{\La(|\om|)-1}{\La(|\om|)+1}\,\frac{\om}{\ol{\om}} \ \mbox{ and } \ \be(r), \ga(r)\to 0 \ \mbox{ uniformly as } \ r\to 0.
$$
By these limits, \eqref{equation-for-cr} gives \eqref{asymptotics-c-affine}.
\end{proof}

\begin{proof}[Proof of Theorem \ref{th:asymptotics-F}]
Thanks to Lemma \ref{lem:asymptotics-F}, it is sufficient to show that
$$
\lim_{r\to 0} c^\cF(f,r)(z)=\lim_{r\to 0} c^\cF(f_a,r)(z),
$$
where $f_a(\zi)$ is the affine function defined in Lemma \ref{prop:equation-for-cr}, with the choice $\om=f(z)$, $\si=f_z(z)$, and $\tau=f_{\ol{z}}(z)$.
\par
Since $f$ is differentiable at $z\in\Om$, we have that
$$
f(\zi)=f(z)+f_z(z)\,(\zi-z)+f_{\ol{z}}(z)\,\ol{(\zi-z)}+o(|\zi-z|) \ \mbox{ as } \ |\zi-z|\to 0.
$$
This means that, for any $\eta>0$, there exists $r_\eta>0$ such that
\begin{equation}
\label{bound-f-fa}
|f(z+r\,\zi)-f_a(z+r\,\zi)|<\eta\,r \ \mbox{ for } \ \zi\in D \ \mbox{ and } \ 0<r<r_\eta.
\end{equation}
\par
Next, for notational convenience, we set $c_r=c^\cF(f,r)(z)$ and $c^a_r=c^\cF(f_a,r)(z)$.
By the characterizations of $c_r$ and $c_r^a$ (see \eqref{fermat-c}), we know that
$$
\frac1{r} \int_{\cS} \left\{H_{\ol{w}}\left(f(z+r\,\zi)-r\,c_r\,\ol{\zi}\right)-H_{\ol{w}}\left(f_a(z+r\,\zi)-r\,c_r^a\,\ol{\zi}\right) \right\} \zi\,dS_\zi=0.
$$
Then, we proceed with the use of the fundamental theorem of calculus, as in the proof of Proposition \ref{prop:equation-for-cr}. This time we set
$$
w=w(t,r;\zi)=f_a(z+r\,\zi)-r\,c_r^a\,\ol{\zi}+t\,[f(z+r\,\zi)-f_a(z+r\,\zi)-r\,(c_r-c_r^a)\,\ol{\zi}]
$$
and obtain:
\begin{multline*}
\int_\cS\int_0^1H_{\ol{w} w}\left(w\right)\left[\frac{f(z+r\,\zi)-f_a(z+r\,\zi)}{r}-(c_r-c_r^a)\,\ol{\zi}\right]\zi\,dt\,dS_\zi+\\
\int_\cS\int_0^1H_{\ol{w} \ol{w}}\left(w\right)\left[\frac{\ol{f(z+r\,\zi)}-\ol{f_a(z+r\,\zi)}}{r}-\ol{(c_r-c_r^a)}\,\zi\right]\zi\,dt\,dS_\zi=0.
\end{multline*}
\par
Now, we use the fact that $H(w)=F(|w|)$ and follow the arguments of the proof of Lemma \ref{lem:asymptotics-F}. After similar manipulations, from \eqref{bound-f-fa} we finally gain the inequality
$$
|c_r-c_r^a|\le C\,\eta,
$$
where $C=C(\La^-, \La^+)$ is a positive constant. This means that $|c_r-c_r^a|\to 0$ as $r\to 0$, since
$$
\limsup_{r\to 0}|c_r-c_r^a|\le C\,\eta
$$
and $\eta$ is arbitrary. The proof is complete.
\end{proof}

\begin{rem}[{\sl Quasiregular} functions]
\label{quasireg}
\rm{
Note that from \eqref{general-eq-f} it follows that, if $c^\cF(f,r)(z)=o(1) \ \mbox{ as } \ r\to 0^+ $, then
$$
|f_{\ol{z}}(z)| \leq \left|\frac{\La(|f(z)|)-1}{\La(|f(z)|)+1}\right|\,|\ol{f}_{\ol{z}}(z)|\leq  \max\left[\frac{\La^+-1}{\La^+ +1},\frac{\frac{1}{\La^-}-1}{\frac{1}{\La^-}+1}\right] |f_z(z)|,
$$
for every $z$ such that $f(z)\neq 0$. Thus, if $f$ is a non-constant function belonging to $W^{1,2}_{\rm{loc}}(\Omega)$, we obtain that $f$ is $K$-quasiregular with
$$
K=\max\left[\La^+,\frac{1}{\La^-}\right]\in [1,+\infty).
$$
(For the definition and properties of $K$-quasiregular functions see the discussions in  \cite[Section 2]{IwMa} and \cite[Section 3]{AlLuRoss}.)
}
\end{rem}

\smallskip

\subsection{The case of $p$-means}
An important special case occurs when we choose $F(s)=s^p/p$, $1<p<\infty$.
Thus, we set $c_p(f,r)=c^\cF(f,r)(z)$ and  call it a \textit{$p$-mean}. Then, we compute that $\La(s)=p-1$, and hence \eqref{general-eq-f} reads as
\begin{equation}
\label{eq-f-p}
f_{\ol{z}}+\frac{p-2}{p}\,\frac{f}{\ol{f}}\,\ol{f}_{\ol{z}}=0 \ \mbox{ in } \ \Om.
\end{equation}

We shall next examine apart the case in which $p=\infty$.

\begin{lem}
\label{lem:infinity}
Let
$$
f_a(\zi)=\om+\si\,(\zi-z)+\tau\,\ol{(\zi-z)} \ \mbox{ for } \ \zi\in\CC,
$$
with $\om, \si, \tau\in\CC$ and $\om\ne 0$. Let $c_\infty(f_a,r)$ be any minimum point of
the function $\cF_\infty:\CC\to [0,\infty)$ defined by
$$
\cF_\infty(c)=\max_{\zi\in \pa D_r(z)} \left|f_a(\zi)-c\,\ol{(\zi-z)}\right|, \ c\in\CC.
$$
Then, it holds that
$$
c_\infty(f_a,r)=\tau+\frac{\om}{\ol\om}\,\ol{\si}+o(1) \ \mbox{ as } \ r\to 0^+.
$$
\end{lem}

\begin{proof}
We fix $\om, \si, \tau\in\CC$ with $\om\ne 0$. For notational convenience, we set $c_r=c_\infty(f_a,r)$.
We have that
$$
\cF_\infty(c)=\max_{\zi\in \cS} \left|\om+r\,[\si\,\zi+(\tau-c)\,\ol{\zi}]\right|
$$
and notice that $c_r$ also minimizes the function
\begin{multline*}
\CC\ni c\mapsto \frac{\cF_\infty(c)^2-|\om|^2}{r}= \\
\max_{\zi\in\cS}\left\{2\,\re \bigl[\ol{\om}\,[\si\,\zi+(\tau-c)\,\ol{\zi}]\bigr]
+r\,\bigl|\si\,\zi+(\tau-c)\,\ol{\zi}\bigr|^2\right\}.
\end{multline*}
Next, we obeserve that
\begin{multline*}
\frac{\cF_\infty(c)^2-|\om|^2}{r}\ge
2\,\max_{\zi\in\cS}\left\{\re \bigl[\ol{\om}\,[\si\,\zi+(\tau-c)\,\ol{\zi}]\bigr]\right\}= \\
2\,\max_{\zi\in\cS}\left\{\re \bigl[\bigl(\om\,\ol{\si}+\ol{\om}\,(\tau-c)\bigr)\ol{\zi}\bigr]\right\}=
2\,|\ol{\om}\,\si+\om\,\ol{(\tau-c)}|\ge2\, |\om|\,(|\tau-c|-|\si|).
\end{multline*}
Thus, we infer that
$$
|\tau-c_r|\le |\si|+\frac{\cF_\infty(c_r)^2-|\om|^2}{2\,r\,|\om|}\le
|\si|+\frac{\cF_\infty(0)^2-|\om|^2}{2\,r\,|\om|},
$$
and the last term, and hence $c_r$, remains bounded for $0<r\le 1$.
\par
Now, as $r\to 0^+$, it is clear that
$$
\frac{\cF_\infty(c)^2-|\om|^2}{r}\to 2\,\max_{\zi\in\cS}\left\{\re \bigl[(\om\,\ol{\si}+\ol{\om}\,(\tau-c)\bigr]\ol{\zi}\right\}=2\,\bigl|\om\,\ol{\si}+\ol{\om}\,(\tau-c)\bigr|,
$$
uniformly in $c$ on compact subsets of $\CC$.
\par
Therefore, any minimum point $c_\infty(f_a,r)$ converges to the unique minimum point
$$
\tau+\frac{\om}{\ol\om}\,\ol{\si}
$$
of the function $\CC\ni c\mapsto 2\,\bigl|\om\,\ol{\si}+\ol{\om}\,(\tau-c)\bigr|$ and this gives the desired conclusion.
\end{proof}

\begin{cor}
\label{th:asymptotics-infty}
Let $f\in C^1(\Om;\CC)$ and let $c_\infty(f,r)$ be any minimum point of the function $\cF_\infty:\CC\to [0,\infty)$ defined by
$$
\cF_\infty(c)=\max_{\zi\in \pa D_r(z)} |f(\zi)-c\,\ol{(\zi-z)}|, \ c\in\CC.
$$
Then, away from the zeroes of $f$, we have that
$$
c_\infty(f,r)=o(1) \ \mbox{ as } \ r\to 0^+
$$
if and only if
$f$ satisfies the equation
\begin{equation}
\label{eq-f-infty}
f_{\ol{z}}+\frac{f}{\ol{f}}\,\ol{f}_{\ol{z}}=0 \ \mbox{ in } \ \Om.
\end{equation}
\end{cor}

\begin{proof}
We observe that, by a Taylor expansion, by choosing $\om=f(z), \si=f_z(z)$, and $\tau=f_{\ol{z}}(z)$ we have that
$$
f(\zi)=f_a(\zi)+o(r) \ \mbox{ as } \ r\to 0^+,
$$
uniformly on any compact subset of $\Om$.
The conclusion then follows by the same argument used in the proof of Lemma \ref{lem:infinity}.
\end{proof}

\begin{rem}
{\rm
Notice that in the limit for $p\to\infty$ from \eqref{eq-f-p} we obtain \eqref{eq-f-infty}.
\par
Also, observe that \eqref{eq-f-infty} can be rewritten as
$$
f\,\pa_{\ol{z}} \log(f\,\ol{f})=0
$$
away from the zeroes of $f$. In other words, we have that $\log |f|$ is a real-valued holomorphic function, and hence is constant (toghether with $|f|$) on 
domains which do not contain zeroes of $f$.
}
\end{rem}

\begin{rem}[$p$-harmonic functions]
\label{pharm}
\rm{
For $\Om\subseteq\CC$, let $u\in C^2(\Om;\RR)$ and assume that $f=2\,\pa_z u=u_x-i\,u_y$. It follows that $2\,f_{\ol{z}}=\De u$ and, in particular, $f_{\ol{z}}$ is real-valued.  This gives that \eqref{eq-f-p} can be rewritten as
\begin{equation}
\label{bojarski-iwaniec}
f_{\ol{z}}= \frac{2-p}{2p} \left(\frac{\ol{f}}{f}\,f_z + \frac{f}{\ol{f}}\,\ol{f}_{\ol{z}} \right)
\end{equation}
at any point for which $f\ne 0$, that is to say at points such that $\na u\ne 0$.
\par
From \cite{BI} we know that $f=2\,\pa_z u$ is a solution of \eqref{bojarski-iwaniec} if and only if
$\De^G_p u= 0$.
Hence, we can conclude that $c_p(\pa_z u,r)=o(1)$  as $r\to 0^+$ if and only if
$\De_p^G u=0$.
\par
Note that, as shown in \cite{BI},  the function defined by $g=|f|^{\sqrt{p-1}-1}f$
is a solution of the Beltrami equation:
$$
g_{\ol{z}}=\frac{1-\sqrt{p-1}}{1+\sqrt{p-1}}\,\frac{\ol{g}}{g}\,g_z,
$$
with the dilation coefficient
$$
\frac{|1-\sqrt{p-1}|}{1+\sqrt{p-1}}<1.
$$
}
\end{rem}

\begin{rem}
{\rm
Note that, for the variational $p$-mean defined in \eqref{variational}, we have that $\La(s)=\Lambda^-=\La^+=p-1$ for every $s>0$.
\par
From Remark \ref{quasireg}, it thus follows  that, if $c_p(f,r)=o(1)$ as $r \to 0^+$, $f\in W^{1,2}_{\rm{loc}}(\Om)$, and the set $\{z\in \Om : f(z)=0\}$ has zero Lebesgue measure, then $f$ is $K$-quasiregular with
$$
K=\max\left[p-1,\frac{1}{p-1}\right].
$$
In particular, from Remark \ref{pharm} we deduce that, if $u$ is $p$-harmonic with non-constant gradient and $\na u\in W^{1,2}_{\rm{loc}}(\Om;\RR^2)$, then $\nabla u$ is $K$-quasiregular (see also \cite[Introduction]{ArLin} and \cite[Section 2]{Aron}).
\par
Recall that, in the plane, the critical points of $p$-harmonic functions are isolated unless $u$ is constant (see \cite{BI} or Corollary 1 of \cite{Man}). This fact in \cite{LinMan} is used to prove an alternative asymptotic mean value property which holds pointwise in all $\Omega$ at least in the range $1<p< p_0 = 9.52520797..$. This range was later extended for all $1<p<\infty$ in \cite{ArrLlo}.
}
\end{rem}

\smallskip

\subsection{A nonlinear asymptotic characterization of holomorphic functions}
\label{subsec:nonlinear-holomorphic}

If we examine \eqref{general-eq-f}, we can see that the transformation
$$
f=\psi(|g|)\,\frac{g}{|g|}
$$
gives that, if
$$
\La(\psi(t))\,\psi'(t)-\frac{\psi(t)}{t}=0,
$$
then $g_{\ol{z}}=0$, away from the zeroes of $g$.
By recalling the definition of $\La$, we see that $F'(\psi(t))=C\,t$  for some positive constant $C$, which we can normalize to $1$. Therefore, we obtain that the function
$$
g=F'(|f|)\frac{f}{|f|}
$$
is holomorphic if $f$ satisfies \eqref{general-eq-f}.
In particular, for equation \eqref{eq-f-p}, we have that
$g=|f|^{p-2} f$
is holomorphic.
\par
We shall now use these remarks to obtain a characterization of holomorphic functions by means of a nonlinear asymptotic mean value property. In order to see this, besides assuming that $F\in C^1([0,\infty))\cap C^2((0,\infty))$ is strictly convex and that $F(0)=F'(0)=0$, we also require that $F$ satisfies the growth condition
\begin{equation*}
\lim_{s\to\infty} \frac{F(s)}{s}=+\infty.
\end{equation*}
Under these assumptions, the so-called \textit{Young or Fenchel conjugate} of $F$ is defined for every $t\ge 0$ by
$$
G(t)=\sup\{t\,s-F(s): s\ge 0\}.
$$
It easily seen that $F'$ is invertible, because strictly increasing and continuous. In addition, it is a known fact that $G'$ is the inverse of $F'$, and hence
\begin{eqnarray}
\label{aus2}
&&F'(G'(t))=t, \quad F''(G'(t))\,G''(t)=1, \nonumber \\
&&G'(F'(s))=s, \quad G''(F'(s))\,F''(s)=1,  \\
&&\La(G'(t))=\frac{G'(t)}{t\,G''(t)}. \nonumber
\end{eqnarray}
\par
We are now in position to state the following companion of Lemma \ref{lem:asymptotics-F}.
\begin{lem}
\label{lem:nonlinear-holomorphic}
Let $F\in C^1([0,+\infty))\cap C^2((0,\infty))$ be a strictly convex function such that $F(0)=F'(0)=0$. Also, let $\La$ be defined in \eqref{Lambda} and satisfy \eqref{Lambda-bounds}. Let $G: [0,\infty)\to [0,\infty)$ be the Young conjugate of $F$.
\par
Suppose that $g$ is differentiable in $\Om$ and let $c^\cJ(g,r)(z)$ be the minimum point in $\CC$ of the function defined in \eqref{defJ}, i.e.
\begin{equation*}
\cJ(c)= \int_{\pa D_r(z)} F\left(\left|\frac{G'(|g(\zi)|)}{|g(\zi)|}g(\zi)-c\,\ol{(\zi-z)}\right|\right) dS_\zi, \ \ c\in\CC.
\end{equation*}
Then, we have that
\begin{equation}
\label{nonlinear-holomorphic}
c^\cJ(g,r)(z) = \frac{2}{1+\La(G'(|g(z)|))}\frac{G'(|g(z)|)}{|g(z)|} g_{\ol{z}}(z) + o(1) \ \mbox{ as } \ r\to 0^+,
\end{equation}
at every $z\in\Om$ such that $g(z)\neq 0$.
\end{lem}

\begin{proof}
It is sufficient to set $f(\zi)=G'(|g(\zi)|)\,g(\zi)/|g(\zi)|$ and apply Lemma \ref{lem:asymptotics-F} to $f$. In fact, we have that
\begin{multline*}
c^\cJ(g,r)=c^\cF(f,r)= \\
f_{\ol{z}}(z)+\frac{\La(|f(z)|)-1}{\La(|f(z)|)+1}\,\frac{f(z)}{\ol{f(z)}}\,\ol{f}_{\ol{z}}(z)+o(1)=
\frac{2\,G'(|g|)\,G''(|g|)}{G'(|g|)+|g|\,G''(|g|)}\,g_{\ol{z}}+o(1),
\end{multline*}
after tedious, but straightforward, calculations involving \eqref{aus2}. Then, \eqref{nonlinear-holomorphic} follows at once.
\end{proof}

\begin{proof}[Proof of Theorem \ref{th:nonlinear-holomorphic}]
It is clear that the theorem is a straightforward consequence of the lemma.
\end{proof}

We conclude this section with the relevant case of a modified $p$-mean.

\begin{cor}[The case of $p$-means]
Let $F(s)=s^p/p$ with  $p>1$ and set $p'=p/(p-1)$.  Let $g:\Om\to\CC$ be differentiable in $\Om$ and let $c_p^\cJ(g,r)(z)$ be the unique minimum point in $\CC$ of the function defined by
$$
\cJ_p(c)=\tfrac1p\int_{\pa D_r(z)} \left||g(\zi)|^{p'-2}g(\zi) -c\,\ol{(\zi-z)}\right|^p dS_\zi, \ \ c\in\CC.
$$
Then, it holds that
\begin{equation}
\label{ceq20}
c_p^\cJ(g,r)(z)= \frac2{p} \,|g(z)|^{p'-2} g_{\ol{z}}(z) + o(1) \ \mbox{ as } \ r\to 0^+,
\end{equation}
at every $z\in\Om$ such that $g(z)\neq 0$. In particular, away from its zeroes, $g$ is holomorphic in $\Om$ if and only if $c_p^\cJ(g,r)(z)= o(1)$ as $r\to 0^+$.
\end{cor}

\begin{proof}
It is easy to check that $G(t)=t^{p'}/p'$, $G'(t)/t=t^{p'-2}$, and $\La(s)=p-1$.
\end{proof}

\section{A classical asymptotic mean value property \\
for nonlinear Cauchy-Riemann systems}
\label{sec:classical-AMVP}

In this section, we shall prove Theorem \ref{th:dpp-theorem}.
We first modify the mean $c^\cF(f,r)$ in the spirit of Proposition \ref{th:new-mean-value} and Corollary \ref{cor:new-mean-value}, in order to obtain an operator acting on $f$ which is appropriate to set up the desired vectorial DPP. Thus, as explained in the introduction, with Remark \ref{rem:projection} in mind, we consider the function defined by
\begin{equation}
\label{def-F-ab}
\cG(a,b)=\int_{\pa D_r(z)} \left[F\bigl(|f(\zi)-a|\bigr)+ F\bigl(|f(\zi)-b\,\ol{(\zi-z)}|\bigr)\right] dS_\zi,
\end{equation}
for $a, b\in\CC$, and denote by $(a^\cG(f,r), b^\cG(f,r))$ the unique minimum point of $\cG$ on $\CC\times\CC$.
\par
Now, it is clear that we can apply Lemma \ref{lem:asymptotics-F} and Theorem \ref{th:asymptotics-F} to $b^\cG(f,r)$, since this is the minimum point in $\CC$ of the function $\CC\ni b\mapsto\cG(a^\cG(f,r),b)$.
Moreover, by similar computations, if $F'(s)=C\,s^\al+o(s^\al)$ as $s\to 0^+$ for some positive constants $C$ and $\al$, we easily infer that
\begin{equation}
\label{asymptotics-a}
a^\cG(f,r)(z)=f(z)+o(r) \ \mbox{ as } \ r\to 0^+,
\end{equation}
since the function $\CC\ni a\mapsto\cG(a, b^\cG(f,r))$ is minimized on $\CC$ at $a^\cG(f,r)$.
Thus, we consider the mean defined in \eqref{def-mu}, i.e.
\begin{equation}
\label{def-mu-G}
\mu^\cG(f,r)(z)=a^\cG(f,r)(z)+r\,b^\cG(f,r)(z),
\end{equation}
and summarize the above remarks in the following result.

\begin{lem}
\label{lem:ab-theorem}
Let $F$ satisfy the assumptions of Theorem \ref{th:asymptotics-F} and  that $F'(s)=C\,s^\al+o(s^\al)$ as $s\to 0^+$ for some positive constants $C$ and $\al$.
Let $f$ be differentiable in $\Om$. Then, it holds that
\begin{multline}
\label{abeq}
\mu^\cG(f,r)(z)= \\
f(z) + r\, \left[ f_{\ol{z}}(z)+\frac{\La(|f(z)|)-1}{\La(|f(z)|)+1}\,\frac{f(z)}{\ol{f(z)}}\,\ol{f}_{\ol{z}}(z)\right] + o(r) \ \mbox{ as } \ r\to 0^+
\end{multline}
at every $z\in\Om$ such that $f(z)\neq 0$.
\end{lem}

\begin{proof}[Proof of Theorem \ref{th:dpp-theorem}]
The theorem is clearly a straightforward corollary of Lemma \ref{lem:ab-theorem}.
\end{proof}

\begin{rem}\label{linkSec2}
{\rm
It is clear that $\mu^\cG(f,r)(z)$ generalizes the mean of Proposition \ref{th:new-mean-value} and Corollary \ref{cor:new-mean-value} (cf. \eqref{new-mean-value}). Indeed, if we select $F(s)=s^2/2$, by applying Theorem \ref{th:dpp-theorem} and keeping in mind \eqref{eq-f-p}, we obtain that $f$ is holomorphic in $\Omega$ if and only if it satisfies the asymptotic mean value property in $\Omega$, i.e.
$$
\mu^\cG(f,r)(z) = a^\cG(f,r)(z)+r\,b^\cG(f,r)(z) = f(z) + o(r) \quad \text{as }r\to 0^+.
$$
With this choice of $F$, it easy check  that the means $a^\cG$ and $b^\cG$ coincide with those mentioned in Remark \ref{rem:projection}.
}
\end{rem}

\begin{rem}
{\rm
Lemma \ref{lem:ab-theorem} informs us that the operator $f\mapsto \mu^\cG(f,r)$ can be taken as the basis of a vectorial DDP:
\begin{equation}
\label{dpp-f}
f^r=\mu^\cG(f^r,r).
\end{equation}
\par
The problem of setting up the appropriate (Dirichlet) boundary conditions for \eqref{dpp-f}  will be considered in a forthcoming paper.
Here, we just observe that the strategy based on comparison principles, which has been employed in the scalar case, may be no longer practicable. However, as an alternative, one can benefit from the continuity properties of the operator $f\mapsto \mu^\cG(f,r)$, which are naturally inherited from the fact that $\mu^\cG(f,r)$ can be thought of as an appropriate projection.
}
\end{rem}

\begin{rem}
{\rm
It is clear that imposing Dirichlet boundary conditions to the DDP \eqref{dpp-f} leads to a Dirichlet problem for the system \eqref{general-eq-f}. It is important to notice that we cannot expect that a solution of such a boundary problem be smooth everywhere in the relevant domain.
\par
This can be easily understood even when we consider the easiest case, namely, that of the Cauchy-Riemann system. In fact, we know that, in a simply connected domain, a (everywhere) holomorphic function is uniquely determined, up to an additive constant, by the values on the boundary of its real or imaginary part. Thus, imposing the boundary values of both real and imaginary parts of a complex-valued function subject to the Cauchy-Riemann equations will lead to a (everywhere smooth) holomorphic solution \textit{only occasionally}.
\par
Therefore, any decent definition of the solution of the Dirichlet problem for the Cauchy-Riemann system should be intended in some (non-smooth) generalized sense. These observations motivate the introduction of the concept of contact solutions and contact mean value property, which we shall consider in the next subsection.
}
\end{rem}

\section{The contact asymptotic mean value property \\
for nonlinear Cauchy-Riemann systems}
\label{subsec:contact-amvp}

In this section, we shall prove Theorem \ref{th:contact-asymptotic-characterization}. While in Theorem \ref{th:dpp-theorem} we considered differentiable functions, here, merely continuous functions will be allowed.

\subsection{Contact solutions }
\label{subsec:gen-contact-solutions}
In order to prepare the proof of Theorem \ref{th:contact-asymptotic-characterization}, we now  recall some notations, definitions and results  from the theory of vector-valued \textit{contact solution} of fully nonlinear PDE systems, recently proposed by N. Katzourakis in \cite{Ka}.
\par
The theory has been set up for general degenerate elliptic \textit{second order} $N\times n$ systems of partial differential equations. Here, we will only report the definitions for \textit{first order} $2\times 2$ systems that are relevant to our aims. In fact, for a function $u:\Om\to\RR^2$, we shall consider the system
\begin{equation}
\label{first-order-system}
\FF(x, u, Du)=0,
\end{equation}
where $\Om$ is a planar open set and  $\FF:\Om\times\RR^2\times\RR^{2\times 2}\to\RR^2$ is a locally bounded map, with variables $x\in\Om$, $\eta\in\RR^2$, and $P\in \RR^{2\times 2}$. Here, we stick to the notations used in \cite{Ka}. In particular, $\RR^{2\times 2}=\RR^2\otimes\RR^2$ denotes the space of $2\times 2$ matrices and $\RR^{2\times 2}_s$ is the subspace of symmetric matrices.
\par
In order to define contact solutions of the system \eqref{first-order-system}, we recall some notations from \cite{Ka}. We first define three operators on (column) vectors $\xi, \eta\in\RR^2$:
\begin{enumerate}[(i)]
\item
$\xi^\top$ is the \textit{transpose} of $\xi$, so that $\xi^\top\eta$ is the scalar product of $\xi$ and $\eta$ or, else, the projection of $\eta$ along $\xi$;
\item
$\xi^\perp=I-\xi\otimes\xi$, so that when $\xi$ is unitary  $\xi^\perp \eta$ is the projection of $\eta$ on the hyperplane orthogonal to $\xi$;
\item
$\xi\vee \eta=\frac12\,(\xi\otimes\eta+\eta\otimes\xi)=\frac12\,(\xi\,\eta^\top+\eta\,\xi^\top)$ is a $2\times 2$ symmetric matrix.
\end{enumerate}
\par
Let $u:\Om\to\RR^2$ be a continuous map. For $\xi\in\SS^1$, the \textit{first contact $\xi$-jet of $u$ at $x\in\Om$} is the set
$$
J^{1,\xi}u(x)=\bigl\{ P\in\RR^2\otimes\RR^2 :  \xi\vee[u(y)-u(x)-P\,(y-x)]\le o(|y-x|) \mbox{ as } y\to x\bigr\}.
$$
Here, we mean that there is a continuous matrix-valued map $M:\RR^2\setminus\{ 0\}\to\RR^{2\times 2}_s$ such that $M(y-x)$ bounds the matrix $\xi\vee[u(y)-u(x)-P\,(y-x)]$ from above (in the sense of matrices), and $|M(z)|/|z|\to 0$ as $|z|\to 0$.
\par
Next, since the system \eqref{general-eq-f} we want to treat has discontinuous coefficients, we need to define the \textit{$\xi$-envelope of $\FF$} as
\begin{equation}
\label{xi-envelope}
\xi^* \FF(x,\eta, P)=\limsup_{r\to 0}\{ \xi^\top \FF(y,\beta, Q): |y-x|+|\beta-\eta|+|Q-P|\le r\}.
\end{equation}
This function is upper semicontinuous.
\par
We are now ready to define a contact solution. Let $\FF:\Om\times\RR^2\times\RR^{2\times 2}\to\RR^2$ be locally bounded. A continuous map $u:\Om\to\RR^2$ is called a \textit{contact solution} of \eqref{first-order-system} in $\Om$
if, for every $x\in\Om$ and $\xi\in\SS^1$, we have that
$$
\xi^* \FF(x,u(x), P)\ge 0  \ \mbox{ for any } \ P\in J^{1,\xi}u(x).
$$
This definition generalizes that of viscosity solution given for scalar equations, which can be recovered by replacing the set $\SS^1$ by $\{+1, -1\}$.
\par
Similarly to the case of viscosity solutions, there is another equivalent definition of contact solutions, which extends the idea of test functions touching from above or below.
In fact, we first define a \textit{cone} with vertex at $x$ and some slope $L>0$ as the function defined by
$$
C_x(y)=L\,|y-x| \ \mbox{ for } \ y\in\RR^2.
$$
Then, we say that a map $\psi\in C^1(\RR^2;\RR^2)$ is a \textit{first contact $\xi$-map for $u$ at $x$} if $\psi(x)=u(x)$ and, for every cone $C_x$, it holds that
\begin{equation}
\label{cone-inequality}
|\xi^\perp (u-\psi)(y)|^2\le C_x(y)\,[-\xi^\top (u-\psi)(y)],
\end{equation}
for any $y$ in some neighborhood of $x$ in $\Om$. Notice that this definition is another way to say that
$$
|\xi^\perp (u-\psi)(y)|^2\le o(|y-x|)\,[-\xi^\top (u-\psi)(y)] \ \mbox{ as } \ y\to x.
$$
Moreover, since the left-hand side of \eqref{cone-inequality} is non-negative and $u(x)=\psi(x)$, the projection $\xi^\top (u-\psi)(y)$ along the line in $\RR^2$
spanned by $\xi$ has a (local) vanishing maximum at $y = x$, i.e.
$$
\xi^\top (u-\psi)\le 0=\xi^\top (u-\psi)(x)
$$
in a neighborhood of $x$.
\par
Thanks to \cite[Theorem 35]{Ka}, one has the equivalence between contact jets and contact maps, in the sense that it turns out that
\begin{multline*}
J^{1,\xi}u(x)=\bigl\{D \psi(x): \psi\in C^1(\Om;\RR^2): \psi(x)=u(x) \mbox{ and} \\
\mbox{\eqref{cone-inequality} holds near $x$ for any cone $C_x$}\bigr\}.
\end{multline*}
As a consequence, one obtains an equivalent definition. In fact, we say that
a continuous map $u:\Om\to\RR^2$ is a contact solution of \eqref{first-order-system} in $\Om$ if, for every $x\in\Om$ and $\xi\in\SS^1$, we have that
$$
\xi^* \FF(x,u(x), D\psi(x))\ge 0,
$$
for any first contact $\xi$-map $\psi\in C^1(\RR^2;\RR^2)$.

\smallskip

\subsection{Contact solutions of nonlinear Cauchy-Riemann systems}
\label{subsec:contact-solutions-CR}
In this section, we proceed to establish a convenient definition of generalized asymptotic mean value property. This will extend the definition of asymptotic mean value property in the viscosity sense, introduced in \cite{MPR1}, to the case of the system \eqref{general-eq-f}.
\par
For consistency, we proceed by adapting the ideas of
Section \ref{subsec:gen-contact-solutions} to the complex-variable framework adopted in this paper. Thus, we consider a first order differential equation
\begin{equation}
\label{general-complex-system}
\FF(f; f_z, f_{\ol{z}})=0,
\end{equation}
where $\FF=\FF_1+i\,\FF_2$ is a complex-valued operator in the variables
$(\om; \si, \tau)\in \CC\times\CC^{1\times 1}$. In the case under scrutiny in this paper (see Eq. \eqref{general-eq-f}), we have that
\begin{equation}
\label{complex-F}
\FF(\om; \si, \tau)=\tau+\frac{\La(|\om|)-1}{\La(|\om|)+1}\,\frac{\om}{\ol{\om}}\,\ol{\si} \ \mbox{ for } \ \om\ne 0.
\end{equation}
\par
We now define the corresponding $\xi$-envelope of $\FF$. If we identify $\xi\in\SS^1$ with the complex number $\xi_1+i\,\xi_2$, for any $(\om; \si, \tau)\in \CC\times\CC^{1\times 1}$ we define:
$$
\xi^* \FF(\om; \si,\tau)=\limsup_{r\to 0}\bigl\{ \,\re\, [\ol{\xi}\,\FF(\om';\si',\tau')]: |\om'-\om|+|\si'-\si|+|\tau'-\tau|\le r\bigr\}.
$$
\par
In the case of \eqref{complex-F}, if $\La(s)$ converges as $s\to 0^+$ to a limit $\La(0^+)\ge 0$, by choosing
$$
\om'=\frac14\,r\,e^{\frac{i}{2}\left[\pi\delta_{\La(0)}+\rm{Arg}(\xi)+\rm{Arg}(\si')\right]} \ \mbox{ with } \ \delta_{\La(0)}=\begin{cases} 0 & \mbox{if } \La(0^+)\geq 1, \\
1 & \mbox{if } \La(0^+)<1,  \end{cases}
$$
we then compute
\begin{equation}
\label{extended-F}
\xi^* \FF(\om; \si,\tau)=
\begin{cases}
\displaystyle \re (\ol{\xi}\,\tau)+\frac{\La(|\om|)-1}{\La(|\om|)+1}\,\re\left[\frac{\om}{\ol{\om}}\,\ol{\xi\,\si}\right] \ &\mbox{ for } \ \om\ne 0, \vspace{5pt}\\
\displaystyle \re (\ol{\xi}\,\tau)+\left|\frac{\La(0^+)-1}{\La(0^+)+1}\right|\,|\si| \ &\mbox{ for }  \om=0.
\end{cases}
\end{equation}
\par
Next, we adapt to the complex setting the definition of first contact $\xi$-jet. Thus, for a continuous function $f:\Om\to\CC$ and $z\in\Om$, we define the set
\begin{multline*}
J^{1,\xi}_\CC f(z)=
\bigl\{ (\si,\tau)\in\CC^{1\times 1} : \\
 \xi\vee[f(\zi)-f(z)-\si\,(\zi-z)-\tau\,\ol{(\zi-z)}]\le o(|\zi-z|) \mbox{ as } \zi\to z\bigr\}.
\end{multline*}
Here, the operator $\vee$ is defined as
$$
\xi\vee \eta= \frac12
\left[
\begin{matrix}
2(\re\xi)(\re\eta) &\im(\xi \eta)\\
\im(\xi \eta) &2(\im\xi)(\im\eta)
\end{matrix}
\right]
\ \mbox{ for } \ \xi, \eta\in\CC.
$$
\par
Let $\Om\subseteq\CC$. According to the theory recalled in Section \ref{subsec:gen-contact-solutions}, a continuous function $f:\Om\to\CC$ is then called a contact solution of \eqref{general-complex-system}--\eqref{complex-F} in $\Om$
if, for every $z\in\Om$ and $\xi\in\SS^1\subset\CC$, we have that
$$
\xi^* \FF(f(z); \si, \tau)\ge 0  \ \mbox{ for any } \ (\si, \tau)\in J^{1,\xi}_\CC f(z).
$$

As seen in Section \ref{subsec:gen-contact-solutions}, we have an equivalent definition, as follows. A continuous function $f:\Om\to\CC$ is a contact solution of \eqref{general-complex-system}--\eqref{complex-F} in $\Om$ if it holds that
$$
\xi^* \FF(f; \psi_z, \psi_{\ol{z}}) \ge 0 \ \mbox{ at } \ z,
$$
for any $\xi\in\SS^1\subset \CC$ and $(z, \psi)\in\Om\times C^1(\Om; \CC)$, with $\psi(z)=f(z)$ and $\psi(z)\ne 0$, such that
\begin{equation}
\label{cone-ine-comp}
|\xi^\perp (f-\psi)(\zi)|^2\le C_z(\zi)\,\left[-\re\left[\ol\xi (f-\psi)(\zi)\right]\right],
\end{equation}
near $z$ for any cone $C_z$.
Here,
$$
\xi^\perp\eta= \eta-\re(\ol{\xi}\,\eta)\,\xi \ \mbox{ for any $\eta$ in $\CC$}.
$$

\par
In the next section, we shall modify this definition, aiming at a characterization of it by a generalized AMVP based on our nonlinear mean \eqref{def-mu-G}.

\smallskip

\subsection{Nonlinear Cauchy-Riemann systems and the CAMVP }
\label{subsec:CAMVP}
In this section, we shall discuss on possible ways to extend the characterizations by means of AMVP, obtained in \cite{IMW} for nonlinear scalar equations, to the case of the nonlinear Cauchy-Riemann systems considered in this paper.
\par
Of course, as Section \ref{sec:mean-value-property} shows, the case of the classical Cauchy-Riemann system presents no difficulties, since all its solutions are smooth.
The situation drastically changes when some nonlinearity is introduced. As in the scalar case, one needs to establish a suitable notion of generalized solutions of the systems under scrutiny which, as seen in Section \ref{sec:classical-AMVP}, present some singularities/degenerations.
\par
Generally, the nonlinear systems under study are not variational, i.e. they are not derived as Euler equations of variational functionals. The use of weak solutions in Sobolev spaces is thus ruled out or at least difficult to pursue.
\par
Nowadays, the theory of viscosity solutions is well established to effectively treat non-variational \textit{scalar} partial differential equations. The recent appearence of its promising extension to systems proposed by N. Katzourakis --- the notion of contact solutions recalled in Section \ref{subsec:gen-contact-solutions} and adapted in Section \ref{subsec:contact-solutions-CR} to the case of nonlinear Cauchy-Riemann systems --- gave us a motivation to investigate on possible characterizations of contact solutions by means of suitable generalized AMVPs.
\par
Thus, in this section, we consider the nonlinear mean $\mu^\cG(f,r)(z)$ defined in \eqref{def-mu}. We have shown in Section \ref{sec:classical-AMVP} that, away from the zeroes of a solution, a notion of classical AMVP can be established rather easily.
\par
In order to extend the notion of AMVP to general continuous complex-valued functions, one may think of defining an extended operator starting from $\mu^\cG(f,r)(z)$. In other words, we may define the \textit{$\xi$-mean envelope} as
\begin{multline*}
\label{def-bb-M}
\xi^* \MM^\cG(\om; \si,\tau)=
\limsup_{r\to 0}\,\bigl\{ r^{-1}\,\re\left[\ol{\xi} \left[ \mu^\cG(f_a,r)(z)-\om'\right]\right]: \\
|\om'-\om|+|\si'-\si|+|\tau'-\tau|\le r\bigr\} \ \mbox{ for every } \ (\om; \si,\tau)\in \CC\times\CC^{1\times 1},
\end{multline*}
where $f_a(\zi)=\om'+\si'\,(\zi-z)+\tau'\,\ol{(\zi-z)}$.
\par
Then, in analogy with what done in Sections \ref{subsec:gen-contact-solutions} and \ref{subsec:contact-solutions-CR}, one idea would be to define some sort of generalized AMVP, based on $\xi^* \MM^\cG$, as follows. A continuous function $f:\Om\to\CC$ satisfies a generalized AMVP in $\Om$ if, for every $z\in\Om$ and $\xi\in\SS^1$, it holds that
$$
\xi^* \MM^\cG(\om; \si,\tau)\ge 0  \ \mbox{ for any } \ (\si, \tau)\in J^{1,\xi}_\CC f(z).
$$
\par
However, this strategy would be successful only if we prove that
$$
\xi^* \MM^\cG(\om; \si,\tau)=\xi^* \FF(\om; \si,\tau) \ \mbox{ for any } \  (\om; \si, \tau)\in \CC\times\CC^{1\times 1},
$$
where $\FF$ is the operator defined in \eqref{complex-F}. By Lemma \ref{lem:asymptotics-F}, this formula certainly holds true when $\om\ne 0$. Unfortunately, when $\om=0$, the formula is no longer true. In fact, an inspection of Proposition \ref{prop:equation-for-cr} and Lemma \ref{lem:asymptotics-F} informs us that, when $\om=\om(r)$ tends to $0$ as a $O(r)$, the functions $\be(r)$ and $\ga(r)$ may not vanish along, as it would be desirable, in view of \ref{extended-F}. This fact is even more evident if we consider the homogeneous case in which $p\,F(s)=s^p$.
\par
Therefore, we must resort to a different strategy. In fact, we propose to follow the ideas contained in \cite{JLM} and \cite{MPR1}.
\par
In \cite{JLM}, the case of the (game theoretic) $p$-Laplace equation is considered and the relevant real-valued second order operator is defined by
\begin{equation}
\label{p-laplace}
\FF(\eta, X)=\tr(X)+(p-2)\,\frac{\lan X\,\eta,\eta\ran}{|\eta|^2},
\end{equation}
where $\eta\in\RR^N$ and $X$ is an $N\times N$ symmetric matrix.  Differently from our case,  this is a \textit{second order} (scalar) elliptic operator. Nevertheless, one can see an analogy with \eqref{complex-F}, if one identifies the complex number $\om$ with the vector $\eta$ and the pair $(\si,\tau)$ with the matrix $X$.
\par
In the case of \eqref{p-laplace}, by choosing $\xi\in\{-1, +1\}=\SS^0$, it is easily seen that the notion of contact solution introduced in \cite{Ka} reduces to the well-known definition of (viscosity) subsolution (for $\xi=+1$) and supersolution (for $\xi=-1$).
Now, in \cite{JLM}, it is pointed out that, in certain instances, the definition of viscosity solution based on the relevant extended operator, which in this case reads as
$$
\xi^* \FF(\eta, X)= \begin{cases} \tr(X)+(p-2)^+ E(X)-(p-2)^- e(X) & \mbox{for }\xi=+1, \\
-\tr(X)+(p-2)^- E(X)-(p-2)^+ e(X) & \mbox{for }\xi=-1,
\end{cases}
$$
does not guarantee effective comparison results. Here, $E(X)$ and $e(X)$ are the maximum and minimum eigenvalues of $X$ and $(p-2)^{\pm}$ are the positive and the negative part of $p-2$ respectively. (See \cite{Ka} for the details concerning second order operators.)
\par
Thus, a different definition of viscosity solution is adopted.
A continuous function in $\Om$ is then declared a viscosity solution of
$$
\De u+(p-2)\,\frac{\lan \na^2 u \na u,\na u\ran}{|\na u|^2}=0 \ \mbox{ in } \ \Om
$$
if, for any pair $(x, \phi)\in\Om\times C^2(\Om)$ such that $\phi$ touches $u$ from above ($+$ in \eqref{sottosopra1}) or $\phi$ touches $u$ from below ($-$ in \eqref{sottosopra1}) at $x$  and $\na\phi(x)\ne 0$, it holds that
\begin{equation}
\label{sottosopra1}
\pm\left[ \De\phi+(p-2)\,\frac{\lan \na^2 \phi \na\phi,\na\phi\ran}{|\na\phi|^2}\right]\ge 0 \ \mbox{ at } \ x.
\end{equation}
We recall that $\phi$ touches $u$ from above at $x$ if $u-\phi$ attains a strict local maximum at $x$ with $u(x)=\phi(x)$.
(Note the analogy of this notion with that of contact map proposed in \cite{Ka} for quite general second order elliptic systems and recalled in Secion \ref{subsec:gen-contact-solutions}.)
\par
This definition of viscosity solution restores the desired comparison properties.
As a by-product of this approach, in \cite{MPR1} viscosity solutions of the game-theoretic $p$-Laplace equation are characterized by means of a generalized AMVP, as follows. (See also \cite{IMW} for an alternative definition of AMVP.)
\par
A continuous function $u$ is said to satisfy in $\Om$ an \textit{AMVP in the viscosity sense} if, for any pair $(x, \phi)\in\Om\times C^2(\Om)$ such that $\phi$ touches $u$ from above ($+$ in \eqref{sottosopra2}) or $\phi$ touches $u$ from below ($-$ in \eqref{sottosopra2}) at $x$ and  $\na\phi(x)\ne 0$, it holds that
\begin{equation}
\label{sottosopra2}
\pm\bigl[\mu_p^r(\phi)(x)-\phi(x)\bigr]\ge o(r^2) \ \mbox{ as } \ r\to 0^+.
\end{equation}
Here, $\mu_p^r(\phi)$ is a relevant mean in \cite{MPR1}.
\par
Therefore, motivated by this analysis, we propose an alternative definition of contact solution of the system \eqref{general-complex-system}--\eqref{complex-F} and a related definition of generalized AMVP.
\par
Let $\Om\subseteq\CC$. We say that a continuous function $f:\Om\to\CC$ is a \textit{contact solution} of the nonlinear Cauchy-Riemann system \eqref{general-eq-f} in $\Om$ (equivalently of \eqref{general-complex-system}--\eqref{complex-F}) if, for every $z\in\Om$ such that $f(z)\ne 0$ and $\xi\in\SS^1\subset\CC$, we have that
$$
\re\bigl[\ol{\xi}\, \FF(f(z); \si, \tau)\bigr]\ge 0  \ \mbox{ for any } \ (\si, \tau)\in J^{1,\xi}_\CC f(z).
$$
\par
Bearing in mind the correspondence between contact jets and contact maps recalled in Secion \ref{subsec:gen-contact-solutions}, we can also say that $f$ is a contact solution in $\Om$ if
$$
\re\bigl[\ol{\xi}\, \FF(f; \psi_z, \psi_{\ol{z}})\bigr]\ge 0 \ \mbox{ at } \ z,
$$
for any $\xi\in\SS^1$ and $(z, \psi)\in\Om\times C^1(\Om; \CC)$, with $\psi(z)=f(z)$ and $\psi(z)\ne 0$, such that \eqref{cone-ine-comp} holds near $z$ for any cone $C_z$.
\par
Associated to this definition of contact solution, we propose our generalized AMVP.
\par
We say that $f$ satisfies the \textit{contact asymptotic mean value property (CAMVP)} in $\Om$ if, for every $z\in\Om$ such that $f(z)\ne 0$ and $\xi\in\SS^1\subset\CC$, we have that
$$
\re\left\{\ol{\xi} \left[ \mu^\cG(f_a,r)(z)-f(z)\right]\right\}\ge o(r) \ \mbox{ as } \ r\to 0^+,
$$
for every $(\si, \tau)\in J^{1,\xi}_\CC f(z)$. Here, the operator $\mu^\cG$ is that defined in \eqref{def-mu-G} and $f_a(\zi)=f(z)+\si\,(\zi-z)+\tau\,\ol{(\zi-z)}$.
\par
Equivalently, $f$ satisfies the (CAMVP) in $\Om$ if
$$
\re\left\{\ol{\xi} \left[ \mu^\cG(\psi,r)(z)-\psi(z)\right]\right\}\ge o(r) \ \mbox{ as } \ r\to 0^+,
$$
for any $\xi\in\SS^1$ and $(z, \psi)\in\Om\times C^1(\Om; \CC)$, with $\psi(z)=f(z)$ and $\psi(z)\ne 0$, such that \eqref{cone-ine-comp} holds near $z$ for any cone $C_z$.
\par
We are finally ready to prove our characterization contained in Theorem \ref{th:contact-asymptotic-characterization}. We restate it here in detail, for the sake of clarity.

\begin{thm}
Let $F\in C^1([0,\infty))\cap C^2((0,\infty))$ be a strictly convex function such that  $F(0)=0$ and $F'(s)=C\,s^\al+o(s^\al)$ as $s\to 0^+$ for some positive constants $C$ and $\al$.
Set
$$
\La(s)=\frac{s\,F''(s)}{F'(s)} \ \mbox{ for } \ s>0
$$
and assume that, for some constants $0<\La^-\le\La^+$, it holds that
$$
\La^-\le\La(s)\le\La^+ \ \mbox{ if } \ s>0.
$$
\par
A continuous function $f:\Om\to\CC$ is a contact solution of \eqref{general-eq-f}, i.e.
$$
f_{\ol{z}}+\frac{\La(|f|)-1}{\La(|f|)+1}\,\frac{f}{\ol{f}}\,\ol{f}_{\ol{z}}=0 \ \mbox{ in } \ \Om,
$$
if and only if $f$ satisfies the CAMVP in $\Om$.
\end{thm}

\begin{proof}
We first notice that
$$
a^\cG(f_a,r)=f(z)+a^\cG(f_a-f(z),r)
$$
and  we know from \eqref{asymptotics-a} that
$a^\cG(f_a-f(z),r)=o(r)$ as $r\to 0^+$.
Hence, the desired conclusion clear follows by an application of Lemma \ref{lem:asymptotics-F}.
\end{proof}

\appendix

\section{The proof of Proposition \ref{th:new-mean-value}}

In this appendix, we will carry out the proof of Proposition \ref{th:new-mean-value}. We begin by showing that the mean value property \eqref{new-mean-value} entails regularity.
\begin{lem}
\label{lem:regularity}
Let $\Om$ be an open subset of $\CC$ and, for fixed $r>0$, suppose that $\Om$ contains at least a closed ball of radius $2r$.
Set $\Om_r=\{z\in\Om: \ol{D_r(z)}\subset\Om\}$.
\par
Let $f:\Om\to\CC$ be an essentially bounded function on $\Om$ such that
\eqref{new-mean-value} holds
for any disk $D_r(z)$ with $\ol{D_r(z)}\subset\Om$.
Then, $f\in C^\infty(\Om_{2r})$.
\end{lem}

\begin{proof}
We proceed as in \cite{MM} with some adjustments. Set:
$$
j_r(\zi)=\frac1{\pi\,r^2}\left(1-\frac{2\,\zi}{r}\right) \cX_{D_r}(\zi).
$$
Then, by extending $f$ as $f\,\cX_\Om$ to $\CC$, \eqref{new-mean-value} reads as a convolution:
$$
f(z)=\int_\Om f(\zi)\,j_r(z-\zi) \,dA_\zi=\int_\CC f(\zi)\,j_r(z-\zi) \,dA_\zi.
$$
Take $z_1, z_2\in\Om_r$ sufficiently close to one another, say $|z_1-z_2|<r$. We have that
\begin{multline*}
|f(z_1)-f(z_2)|\le \int_\CC |j_r(z_1-\zi)-j_r(z_2-\zi)||f(\zi)|\,dA_\zi\le \\
\nr f\nr_\infty\,\int_\CC |j_r(z_1-\zi)-j_r(z_2-\zi)|\,dA_\zi\le
C\,\nr f\nr_\infty\,\frac{|z_1-z_2|}{r},
\end{multline*}
for some numerical constant $C>0$. The last inequality is obtained by standard manipulations and by observing that
$$
\int_\CC\bigl|\cX_{D_r(z_1)}-\cX_{D_r(z_2)}\bigr|\,dA_\zi=|D_r(z_1)\triangle D_r(z_2)|\le 4r\,|z_1-z_2|.
$$
Here, $D_r(z_1)\triangle D_r(z_2)$ denotes the symmetric difference of the two balls.
If instead $|z_1-z_2|\ge r$, then we easily infer that
$$
|f(z_1)-f(z_2)|\le 2\,\nr f\nr_\infty\le 2 \nr f\nr_\infty \frac{|z_1-z_2|}{r}.
$$
Thus, $f$ is Lipschitz continuous in $\Om_r$, and hence is almost everywhere differentiable, by Rademacher's theorem.
\par
Now, fix $z\in\Om_{2r}$. Since $\Om_{2r}$ is open, $z+t$ is in $\Om_{2r}$, for any  sufficiently small real-valued increment $t$. Then we have that
$$
\frac{f(z+t)-f(z)}{t}=\int_\CC j_r(\zi)\,\frac{f(z-\zi+t)-f(z-\zi)}{t}\,dA_\zi.
$$
Since the integrand is bounded by a constant independent of $t$ and converges almost everywhere to $j_r(\zi)\,\pa_x f(z-\zi)$, by the Dominated Convergence Theorem we get that
$$
\pa_x f(z)=\int_{\CC} j_r(\zi)\,\pa_x f(z-\zi)\,dA_\zi.
$$
This holds at every point $z\in\Om_{2r}$. By choosing a purely imaginary increment $i\,t$, we obtain a similar formula for $\pa_y f(z)$.
Thus, we have proved that $f$ has partial derivatives in $\Om_{2r}$. These must be locally Lipschitz continuous, since we have proved that they satisfy the same integral equation as $f$.
Hence, $f$ is of class $C^1(\Om_{2r})$.
\par
The desired conclusion then follows by iterating the same argument.
\end{proof}

\smallskip

\begin{proof}[Proof of Proposition \ref{th:new-mean-value}]
(i) Suppose that \eqref{new-mean-value} holds for any disk $D_r(z)$ with $\ol{D_r(z)}\subset\Om$. Lemma \ref{lem:regularity} tells us that $f\in C^\infty(\Om)$. Thus, we know that
$$
f(\zi)=f(z)+f_\zi(z)\,(\zi-z)+f_{\ol{\zi}}(z)\,\ol{(\zi-z)}+o(|\zi-z|),
$$
uniformly as $\zi\to z$ on a neighborhood of $z$. Therefore, we can write that
\begin{multline*}
\frac1{|D_r(z)|}\int_{D_r(z)} f(\zi)\,dA_\zi= \\
\frac1{|D_r(z)|}\int_{D_r(z)}\left[f(z)+ f_\zi(z)\,(\zi-z)+f_{\ol{\zi}}(z)\,\ol{(\zi-z)}+o(|\zi-z|)\right]dA_\zi= \\
f(z)+o(r) \ \mbox{ as } \ r\to 0^+,
\end{multline*}
since
$$
\int_{D_r(z)} (\zi-z)\,dA_\zi=\int_{D_r(z)} \ol{(\zi-z)}\,dA_\zi=0.
$$
Similarly, we have that
$$
\frac{2}{r^2}\,\frac1{|D_r(z)|}\int_{D_r(z)} f(\zi)\,(\zi-z)\,dA_\zi=f_{\ol{\zi}}(z)+o(1) \ \mbox{ as } \ r\to 0^+,
$$
since
$$
\int_{D_r(z)} (\zi-z)^2\,dA_\zi=0 \ \mbox{ and } \ \frac1{|D_r(z)|}\int_{D_r(z)}|\zi-z|^2 dA_\zi=\frac12\,r^2.
$$
All in all, \eqref{new-mean-value}, the two asymptotic formulas, and some calculations  give that
$f_{\ol{\zi}}(z)=o(1)$  as $r\to 0^+$,
and hence $f_{\ol{\zi}}(z)=0$. Since $z$ can be chosen arbitrarily in $\Om$, we conclude that $f$ is holomorphic in $\Om$.
\par
(ii) Notice that holomorphic functions are characterized by the \textit{Cauchy Integral Theorem}, i.e. $f$ is holomorphic in a simply connected domain $\Om$ if and only if
$$
\int_\ga f(\zi)\,d\zi=0,
$$
for any closed rectifiable curve $\ga$ contained in $\Om$ (see \cite{Ma}). In particular, if $f$ is holomorphic in $\Om$, we have that
$$
\int_{\pa D_r(z)} f(\zi)\,(\zi-z)\,dS_\zi=-i\,r \int_{\pa D_r(z)} f(\zi)\,d\zi=0,
$$
and hence
$$
\int_{D_r(z)} f(\zi)\,(\zi-z)\,dA_\zi=\int_0^r\left(\int_{\pa D_\rho(z)} f(\zi)\,(\zi-z)\,dS_\zi\right) d\rho=0,
$$
for any disk $D_r(z)$ with $\ol{D_r(z)}\subset\Om$.
Since $f$ is also harmonic in $\Om$, we have that
$$
f(z)=\frac1{|D_r(z)|}\int_{D_r(z)} f(\zi)\,dA_\zi,
$$
and hence we obtain \eqref{new-mean-value}.
\end{proof}

\section*{Acknowledgements}
R. Durastanti has been  supported by the Italian Ministry of University and Research under PON “Ricerca e Innovazione” 2014-2020 (PON R\&I) - AZIONE IV.6 – Contratti di Ricerca su tematiche Green – CUP E65F21003200003.
\par
R. Magnanini is partially supported by the PRIN grant n. 201758MTR2 of the Italian Ministry of University (MUR).
\par
Both authors have also been supported by the Gruppo Nazionale per l'Analisi Matematica, la Probabilit\`a e Applicazioni (GNAMPA) of the Istituto Nazionale di Alta Matematica (INdAM).

\end{document}